\documentclass[12pt]{article}

\usepackage{caption}
\DeclareCaptionLabelSeparator{twospace}{. }
\usepackage{indentfirst}
\usepackage{mathrsfs}
\usepackage{enumerate}
\usepackage[a4paper,textwidth=6.3in,textheight=8.7in]{geometry}  % matches e-jc style sheet
\usepackage[ruled,linesnumbered,vlined]{algorithm2e}
\usepackage{amsmath,amsthm,tikz,float}

\usepackage[colorlinks=true,citecolor=black,linkcolor=black,urlcolor=blue]{hyperref}
\usepackage{mathrsfs}
\usepackage{amssymb}
\usepackage[shortlabels]{enumitem}
\setenumerate[1]{itemsep=0pt,partopsep=0pt,parsep=\parskip,topsep=0pt}
\usepackage{amsmath,amsthm,tikz,float,amssymb}

\newtheorem{theorem}{Theorem}[section]
\newtheorem{lemma}[theorem]{Lemma}
\newtheorem{definition}[theorem]{Definition}
\newtheorem{observation}[theorem]{Observation}
\newtheorem{proposition}[theorem]{Proposition}
\newtheorem{conjecture}[theorem]{Conjecture}
\newtheorem{question}[theorem]{Question}

\newtheorem{claim}{Claim}
\begin{document}

\def\tfn{\footnote{Corresponding author.}}

\title{Generalized power domination in claw-free regular graphs\thanks{Supported in part by National Natural Science Foundation of China (No. 11871222) and Science and Technology Commission of Shanghai Municipality (Nos.
18dz2271000, 19JC1420100)}}
\author{Hangdi Chen \ \ \ \  Changhong Lu\tfn\ \ \ \  Qingjie Ye\\
School of Mathematical Sciences,\\
Shanghai Key Laboratory of PMMP,\\
East China Normal University,\\
Shanghai 200241, P. R. China\\
\\
Email: 471798693@qq.com\\
Email: chlu@math.ecnu.edu.cn\\
Email: mathqjye@qq.com}
\date{}
\maketitle

\begin{abstract}
In this paper, we  give a series of couterexamples to negate a conjecture and hence answer an open question  on the  $k$-power domination  of regular graphs (see  [P. Dorbec et al., SIAM J. Discrete Math., 27 (2013), pp. 1559-1574]).  Furthermore, we focus on the study of $k$-power domination of claw-free graphs. We show that for $l\in\{2,3\}$ and $k\ge l$, the $k$-power domination number of a  connected claw-free $(k+l+1)$-regular  graph on $n$ vertices is at most $\frac{n}{k+l+2}$, and this bound is tight.
\end{abstract}



\section{Introduction}
In this paper, we only consider simple graphs. Let $G=(V(G),E(G))$ (abbreviated as $G=(V,E)$) be a graph. The \emph{open neighborhood} $N_G(v)$ of a vertex $v$ consists of the vertices adjacent to $v$ and its \emph{closed neighborhood} is $N_G[v]=N_G(v)\cup\{v\}$. The \emph{open neighborhood} of a subset $S\subseteq V$ is the set $N_G(S)=\bigcup_{v\in S} N_G(v)$ and its \emph{closed neighborhood} is $N_G[S]=N_G(S)\cup S$. The \emph{degree} of a vertex $v$, denoted $d_G(v)$, is the size of its open neighborhood $|N_G(v)|$. Let $v$ be a vertex of $G$ and $F$ be a subset of $V$. We denote $N_F(v)=N_G(v)\cap F$, $N_F[v]=N_G[v]\cap F$ and $d_F(v)=|N_G(v)\cap F|$.  A graph $G$ is $k$-\emph{regular} if $d_G(v) = k$ for every vertex $v \in V$.  If the graph $G$ is clear from the context, we will omit the subscripts $G$ for convenience. The complete bipartite graph with partite sets of cardinality $i$ and $j$ is denoted by $K_{i,j}$. A \emph{claw-free} graph is a graph that does not contain a claw, i.e. $K_{1,3}$, as an induced subgraph. For a set $S\subseteq V$, we let $G[S]$ denote the subgraph induced by $S$. We say a subset $S\subseteq V$ is  a \emph{packing} if the vertices in $S$ are pairwise at distance at least three apart in $G$.
%We say a subset of $V(G)$ is an \emph{independent set} if no two vertices of the set are adjacent in $G$.
%Let $d(u,v)$ be the \emph{distance} of $u$ and $v$ in graph $G$.
%For two graphs $G=(V,E)$ and $G'=(V',E')$, let $G\cup G'=(V\cup V',E\cup E')$ and $G\cap G'=(V\cap V',E\cap E')$. If $V\cap V'=\emptyset$, then $G$ and $G'$ are called \emph{disjoint}.

%As defined in \cite{favaron2000total}, a set $S$ of vertices in a graph $G$ is called a \emph{total domination set}  (abbreviated as TDS) of $G$  if every vertex of $G$ is adjacent to some vertex in $S$.  The minimum cardinality of a TDS of $G$ is the \emph{total domination number} of $G$, denoted by $\gamma_t(G)$.

Electric power systems must be monitored continually. One way of monitoring these systems is to place phase measurement units (PMUs) at selected locations. Since the cost of a PMU is very high, it is desirable to minimize the number of PMUs. The authors of \cite{baldwin1993power,mili1991phasor}  introduced power domination to model the problem of monitoring electrical systems. Then, the problem was formulated as a graph theoretical problem by Haynes et al. in \cite{haynes2002domination}. Some additional propagation in power domination is using the Kirschoff's laws in electrical systems. The definition of power domination was simplified to the following definition independently in \cite{dorbec2008power,Dorfling2006A,guo2008improved,liao2005power}, which originally asked the systems to monitor both edges and vertices.

\begin{definition}\label{def1}(Power Dominating Set).
Let $G=(V,E)$ be a graph. A subset $S$ of $V$ is a power dominating set (abbreviated
as PDS) of $G$ if and only if all vertices of $V$ are observed either by Observation Rule 1 (abbreviated
as OR 1) initially or by Observation Rule 2 (abbreviated
as OR 2) recursively.%can be recursively observed by the following rules.

{\bf OR 1.} all vertices in  $N_G[S]$ are observed initially.

{\bf OR 2.} If an observed vertex $v$ has all neighbors  observed except one neighbor $u$, then  $u$ is observed (by $v$).
\end{definition}

The \emph{power domination number} $\gamma_p(G)$ is the minimum cardinality of a PDS of $G$. The power domination problem is known to be NP-complete (see \cite{aazami2010domination,aazami2009approximation,guo2008improved,haynes2002domination}).  Linear-time algorithms for this problem were presented for trees, interval graphs and block graphs (see \cite{haynes2002domination,liao2005power,xu2006power}). The Nordhaus-Gaddum problems for power domination were investigated in \cite{benson2018nordhaus} and  parameterized results were given in \cite{kneis2006parameterized}. The exact values of the power domination numbers of some special graphs were studied in   \cite{dorbec2008power,Dorfling2006A}. The upper bounds
for the power domination numbers of regular graphs were investigated (see, for example, \cite{Min2006Power,xu2011power}).

Chang et al. \cite{chang2012generalized} generalized power domination to $k$-power domination. In here, we use a definition of monitored set to define $k$-power dominating set.
\begin{definition} (Monitored Set).
Let $G=(V,E)$ be a graph, let $S\subseteq V$, and let $k\ge 0$ be an integer. We define the sets $(P^i_G(S))_{i\ge 0}$ of vertices monitored by $S$ at step $i$ by the following rules:

(1) $P_G^{0}(S)= N_G[S]$;

(2) $P_G^{i+1}(S)=\cup \{N_G[v]: v\in P_G^{i}(S)$ such that $|N_G[v] \setminus P_G^{i}(S)|\leq k\}$.
\end{definition}
It is clear that $P_G^{i}(S)\subseteq P_G^{i+1}(S)\subseteq V$ for any $i$.  If $P_G^{i_0}(S)=P_G^{i_0+1}(S)$ for some $i_0$, then $P_G^{j}(S)=P_G^{i_0}(S)$ for every $j\ge i_0$ and we accordingly define $P^{\infty}_G(S)=P_G^{i_0}(S)$.

\begin{definition}\label{def1}($k$-Power Dominating Set).
Let $G=(V,E)$ be a graph, let $S\subseteq V$, and let $k\ge 0$ be an integer. If $P^{\infty}_G(S)=V$, then  $S$ is called a $k$-power dominating set of $G$, abbreviated $k$-PDS. The $k$-power domination number of $G$, denoted by $\gamma_{p,k}(G)$, is the minimum cardinality of a $k$-PDS in $G$.
\end{definition}

The $k$-power domination problem is known to be NP-complete for chordal graphs and bipartite graphs \cite{chang2012generalized}. Linear-time algorithms for this problem were presented for trees \cite{chang2012generalized} and block graphs \cite{Chao2016}.  The bounds for the $k$-power domination numbers in regular graphs were obtained in \cite{chang2012generalized,Dorbec2013Generalized}. The relationship between the $k$-forcing and the $k$-power domination numbers of a graph was given in \cite{ferrero2018relationship}. The authors of \cite{dorbec2014generalized} studied the exact values for the $k$-power domination numbers in Sierpi$\acute{n}$ski graphs.

If $G$ is a connected $(k+1)$-regular graph, then $\gamma_{p,k}(G)=1$. Some scholars began to study the $k$-power domination number of $(k+2)$-regular graphs. Zhao et al. \cite{Min2006Power} showed that if $G$ is a $3$-regular claw-free graph on $n$ vertices, then $\gamma_{p,1}(G)\le \frac{n}{4}$. Chang et al. \cite{chang2012generalized} generalized this result to  $(k+2)$-regular claw-free graphs.  Dorbec et al. \cite{Dorbec2013Generalized} removed the claw-free condition and show that $\gamma_{p,k}(G)\le \frac{n}{k+3}$ if $G$ is a  $(k+2)$-regular graph on $n$ vertices. Moreover, they presented  the following conjecture and question.

\begin{conjecture}\label{DorbConj}(\cite{Dorbec2013Generalized})
For $k\ge 1$ and $r\ge 3$, if $G\not\cong K_{r,r}$ is a connected $r$-regular graph of order $n$, then $\gamma_{p,k}(G)\le \frac{n}{r+1}$.
\end{conjecture}
\begin{question}\label{DorbQues}(\cite{Dorbec2013Generalized})
For $r\ge 3$, let $G\neq K_{r,r}$ is a connected $r$-regular graph of order $n$. Determine the smallest positive value, $k_{\min}(r)$, of $k$ such that $\gamma_{p,k}(G) \le \frac{n}{r + 1}$.
\end{question}

The result of Dorbec et al. in \cite{Dorbec2013Generalized} implies that Conjecture \ref{DorbConj} holds for $k=1$ and $r=3$ and $k_{min}(r)\le r-2$. Recently,  Lu et al. \cite{Lu2018} showed that Conjecture \ref{DorbConj} does not always hold for each even $r\ge 4$ and $k=1$.
In this paper, we  show that $k_{\min}(r)=r-2$ for $r\ge 3$ and negate Conjecture \ref{DorbConj} for each  $r\ge 4$ and $1\le k\le r-3$.  We  also show that there exists a series of claw-free $r$-regular graphs $G$ of order $n$ such that $\gamma_{p,k}(G)>\frac{n}{r}$  if $k<\lfloor\frac{r}{2}\rfloor$. But Conjecture \ref{DorbConj} may hold for  claw-free $r$-regular graphs if $k\ge \lfloor\frac{r}{2}\rfloor$.  The following theorem is the main result in this paper.

\begin{theorem}\label{MainResult}
For $l\in\{2,3\}$ and $k\ge l$, if $G$ is a  connected claw-free $(k+l+1)$-regular  graph of order $n$, then $\gamma_{p,k}(G)\le \frac{n}{k+l+2}$ and the bound is tight.
\end{theorem}

\section{Counterexamples}

 Motivated by the concept of a fort proposed in \cite{bozeman2019restricted}, we define the  concept of a $k$-fort, which is a natural generalization of a fort.
%In order to prove the lower bounds in this section, we introduce to the $k$-fort. In \cite{bozeman2019restricted}, they proposed the definition when $k=1$.
\begin{definition}($k$-fort).
    For an integer $k\ge 1$, a $k$-fort of a graph $G$ is a nonempty set $F \subseteq V$ such that each vertex of $N_G(F)\backslash F$ is adjacent to at least $k+1$ vertices in $F$.
    \end{definition}
 If $F$ is a $k$-fort of $G$, then $|F|\ge k+1$.  We immediately obtain the following proposition. %Thus, we omit the proof.
\begin{proposition}\label{k-fort}
    Let $G=(V,E)$ be a graph and $F$ be a $k$-fort of $G$. If $S$ is a $k$-PDS of $G$, then $S \cap N_G[F] \neq \emptyset$.
\end{proposition}

%\subsection{The answer of Question \ref{DorbQues}}

 %Let  $k_{min}(r)$ be the smallest positive value defined in Question \ref{DorbQues}. By Observation \ref{fan1}, $k_{min}(r)\ge r-2$. Since Dorbec et al. \cite{Dorbec2013Generalized} showed that $k_{min}(r)\le r-2$, the answer of Question \ref{DorbQues} is $k_{min}(r)=r-2$.

\begin{observation}\label{fan1}
    For each $r\ge 4$ and $q\ge 2$, there exists a connected $r$-regular graph $D_{r,q}\neq K_{r,r}$ of order $n=2qr$ such that $\gamma_{p,r-3}(D_{r,q})=2q=\frac{n}{r}>\frac{n}{r+1}$.
\end{observation}
\begin{proof}
    We define the graph $D_{r,q}$ as follows: Take $q$ disjoint copies $D_i\cong K_{r,r}-x_iy_i$, where $x_i, y_i\in V(K_{r,r})$ and $i\in\{1,2,\cdots q\}$. Then add  edges $y_ix_{i+1}$ for each $i\in\{1,2,\cdots,q\}$, where $x_{q+1} = x_1$ (see Figure \ref{drq}).
Suppose that $T=\bigcup_{i=1}^q \{x_i,y_i\}$ and $k=r-3$. It is clear that $T$ is a $k$-PDS of $D_{r,q}$. Then, we have $\gamma_{p,k}(D_{r,q})\le |T|\le 2q$. Now, we show $\gamma_{p,k}(D_{r,q})\ge 2q$. Let $S$ be a $k$-PDS of $D_{r,q}$. Assume that $(X_i,Y_i)$ is the bipartition of $D_i$, where $x_i\in X_i$, $y_i\in Y_i$ and $i\in\{1,2,\cdots,q\}$. We claim that $|S\cap V(D_i)|\ge 2$ for each $i\in\{1,2,\cdots,q\}$. Otherwise, without loss of generality, suppose that $|S\cap V(D_1)|\le 1$ and  $S\cap Y_1=\emptyset$. Then $F=X_1\backslash (S\cup \{x_1\})$ is a $k$-fort and $N_{D_{r,q}}[F]\cap S=\emptyset$, contradicting Proposition \ref{k-fort}.
\end{proof}

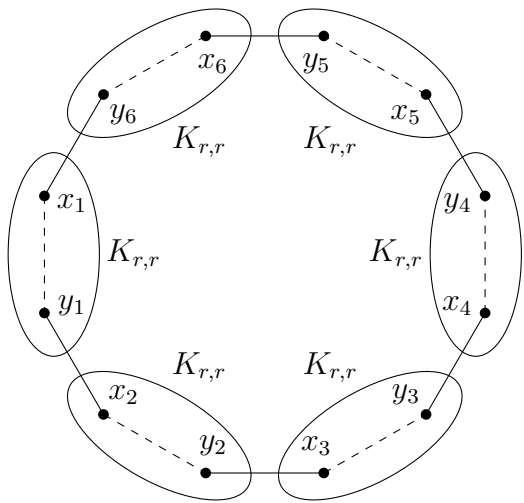
\begin{figure}[H]
    \centering
    \begin{tikzpicture}
        \foreach \i in {1,...,6}
        {
            \coordinate (x\i) at (105+60*\i:3);
            \coordinate (y\i) at (135+60*\i:3);
            \node at (105+60*\i:1.75*1.5) {$x_\i$};
            \node at (135+60*\i:1.75*1.5) {$y_\i$};
            \node at (120+60*\i:1.15*1.5) {$K_{r,r}$};
            \draw[dashed] (x\i) -- (y\i);
            \fill (x\i) circle[radius=2pt];
            \fill (y\i) circle[radius=2pt];
            \draw[rotate=60*\i+30] (0,1.85*1.5) ellipse [x radius=0.9*1.5,y radius=0.4*1.5];
        }
        \foreach \i[evaluate={\j=int(\i+1)}] in {1,...,5}
        {
            \draw (y\i) -- (x\j);
        }
        \draw (y6) -- (x1);

    \end{tikzpicture}
    \caption{The graph $D_{r,6}$}\label{drq}
\end{figure}

By Observation \ref{fan1}, we know that   Conjecture \ref{DorbConj} does not always hold for each $r\ge 4$ and $1\le k\le r-3$, and hence $k_{\min}(r)=r-2$ for $r\ge 3$.  A natural problem is whether  $\frac{n}{r}$ is always the upper bound of $\gamma_{p,k}(G)$ in Conjecture \ref{DorbConj}. We will discuss this problem using the relation between $k$-power domination and total domination in regular graphs.

%From Observation \ref{fan1} and Conjecture \ref{DorbConj},
%the readers may propose a natural problem that is Question \ref{upQues}.
%\begin{question}\label{upQues}
%Is $\frac{n}{r}$ always the upper bound of $\gamma_{p,k}(G)$ in Conjecture \ref{DorbConj}?
%\end{question}
%Unfortunately, $\frac{n}{r}$ is not always the upper bound of $\gamma_{p,k}(G)$ in Conjecture \ref{DorbConj}. We will respond to Question \ref{upQues} in Section 2.2 and Section 2.3.
%\subsection{Relation between $k$-power domination and total domination }
%In this part, we show the relation between $k$-power domination and total domination by presenting Lemma \ref{tp}. Then, we use Lemma \ref{tp} to construct a series of graphs such that $\gamma_{p,k}(G)=\frac{3n}{2r}>\frac{n}{r}$ in Conjecture \ref{DorbConj} (see Observation \ref{Fkq}).

%As defined in \cite{favaron2000total},
  A set $S$ of vertices in a graph $G$ is called a \emph{total domination set}  (abbreviated as TDS) of $G$  if every vertex of $G$ is adjacent to some vertex in $S$. The minimum cardinality of a TDS of $G$ is the \emph{total domination number} of $G$, denoted by $\gamma_t(G)$. Now we present the following  observation.

\begin{observation}\label{tp}
    For each  $k\ge 1$ and $r\ge 1$, if $G$ is a connected $r$-regular graph of order $n$, then there exists a connected $r'$-regular graph $G'$ of order $n'=(k+2)n$  such that  $r'=(k+2)r$ and $\gamma_{p,k}(G')=\gamma_t(G)$.
\end{observation}

\begin{proof}
    Let $V(G)=\{v_1,v_2,\cdots,v_n\}$. Let $G'$ be the graph constructed from $G$ as follows. Take $n$ disjoint independent sets $V_i=\{v_i^1,v_i^2,\cdots,v_i^{k+2}\}$ corresponding to $v_i$, where $i\in\{1,2,\cdots,n\}$.  For each edge $v_iv_j\in E(G)$, add the edges $v_i^s v_j^q$ for each $s,q\in \{1,2,\cdots,k+2\}$ (see Figure \ref{ftp}).

    Let $S=\{v_{i_1},v_{i_2},\cdots,v_{i_h}\}$ be a TDS of $G$ with $h=\gamma_t(G)$. It is easy to check that $\{v_{i_1}^1,v_{i_2}^1,\cdots,v_{i_h}^1\}$ is a $k$-PDS of $G'$. Hence, $\gamma_{p,k}(G')\le \gamma_t(G)$. On the other hand, let $S'$ be a $k$-PDS of $G'$ with $|S'|=\gamma_{p,k}(G')$.  We can change some vertices of $S'$ such that $|S'\cap V_i|\le 1$ for each $i\in\{1,2,\cdots,n\}$. Otherwise, without loss of generality, assume that $|S'\cap V_1|\ge 2$. If there exists $j\in\{2,3,\cdots,n\}$ such that $S'\cap V_j\neq \emptyset$ and $V_j\subseteq N_G(v_1^1)$, then $S''=(S'\backslash V_1) \cup \{v_1^1\}$ is also a $k$-PDS of $G'$ and $|S''|<|S'|=\gamma_{p,k}(G')$, a contradiction. Now we assume $S'\cap V_j=\emptyset$ for each $V_j\subseteq N_G(v_1^1)$, where $j\in\{2,3,\cdots,n\}$. Let $S''=(S'\backslash V_1) \cup \{v_1^1,v_j^1\}$. Thus, $S''$ is also a $k$-PDS of $G'$ such that $|S''\cap V_1|=1$.  Let $S'=S''$. Hence, we find a $k$-PDS $S'$ of $G'$ such that $|S'\cap V_i|\le 1$ for each $i\in\{1,2,\cdots,n\}$. Let $S=\emptyset$. For each $i\in\{1,2,\cdots,n\}$, if $|S'\cap V_i|=1$, we add $v_i$ to $S$. Then $S$ is a TDS of $G$ with $|S|=\gamma_{p,k}(G')$, implying that $\gamma_{p,k}(G')\ge \gamma_t(G)$.
\end{proof}

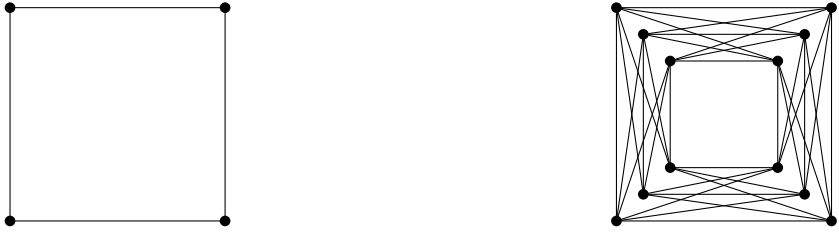
\begin{figure}[H]
    \centering
    \begin{minipage}[t]{0.49\textwidth}
    \centering
    \begin{tikzpicture}
        \foreach \i in {1,2,3,4}
        {
            \coordinate (a\i) at (-45+90*\i:2);
            \fill (a\i) circle[radius=2pt];
        }
        \draw (a1) -- (a2) -- (a3) -- (a4) -- cycle;
    \end{tikzpicture}
    \end{minipage}
    \begin{minipage}[t]{0.49\textwidth}
        \centering
        \begin{tikzpicture}
            \foreach \i in {1,2,3,4}
            {
                \foreach \j in {1,2,3}
                {
                    \coordinate (a\i\j) at (-45+90*\i:\j/2+0.5);
                    \fill (a\i\j) circle[radius=2pt];
                }
            }
            \foreach \i [evaluate={\l=int(\i+1)}] in {1,2,3}
            {
                \foreach \j in {1,2,3}
                {
                    \foreach \k in {1,2,3}
                    {
                        \draw (a\i\j) -- (a\l\k);
                    }
                }
            }
            \foreach \j in {1,2,3}
            {
                \foreach \k in {1,2,3}
                {
                    \draw (a1\j) -- (a4\k);
                }
            }
        \end{tikzpicture}
        \end{minipage}
    \caption{An example of transformation in Observation \ref{tp} for $k=1$}\label{ftp}
\end{figure}

%\begin{observation}\label{Fkq}
 %   For each $k\ge 1$ and $q\ge 1$, there exists a connected $(3k+6)$-regular graph $F_{k,q}$ of order $n=4(k+2)q$ such that $\gamma_{p,k}(F_{k,q})=2q=\frac{3}{2}\frac{n}{3k+6}$.
%\end{observation}
%\begin{proof}
    The authors of \cite{favaron2000total} constructed 3-regular graphs $F_{0,q}$ of order $4q$ such that $\gamma_t(F_{0,q})=2q$ (see Figures \ref{f}-\ref{f04}). By Observation \ref{tp}, we can construct $F_{k,q}$ $(=G')$ from $F_{0,q}$ $(=G)$, and so $\gamma_{p,k}(F_{k,q})=\gamma_t(F_{0,q})=2q=\frac{3}{2}\frac{n'}{3k+6}=\frac{3n'}{2r'}$. Hence, $\frac{n'}{r'}$ is not the upper bound of  $\gamma_{p,k}(G')$ in Conjecture \ref{DorbConj}.
%\end{proof}

\begin{figure}[H]
    \centering
    \centering
    \begin{minipage}[t]{0.38\textwidth}
    \centering
    \begin{tikzpicture}
        \foreach \i in {0}
        {
            \coordinate (z\i1) at (2*\i,1);
            \coordinate (z\i2) at (2*\i+1,1);
            \coordinate (z\i3) at (2*\i+1,0);
            \coordinate (z\i4) at (2*\i,0);

            \fill (z\i1) circle[radius=2pt];
            \fill (z\i2) circle[radius=2pt];
            \fill (z\i3) circle[radius=2pt];
            \fill (z\i4) circle[radius=2pt];

            \draw (z\i2) -- (z\i1) -- (z\i4) -- (z\i3) -- (z\i1);
            \draw (z\i2) -- (z\i4);

        }
        \draw (z02) -- (z03);

    \end{tikzpicture}
    \caption{The graph $F_{0,1}$}\label{f}
    \end{minipage}
    \centering
    \begin{minipage}[t]{0.58\textwidth}
    \centering
    \begin{tikzpicture}
        \foreach \i in {0,1,2,3}
        {
            \coordinate (z\i1) at (2*\i,1);
            \coordinate (z\i2) at (2*\i+1,1);
            \coordinate (z\i3) at (2*\i+1,0);
            \coordinate (z\i4) at (2*\i,0);

            \fill (z\i1) circle[radius=2pt];
            \fill (z\i2) circle[radius=2pt];
            \fill (z\i3) circle[radius=2pt];
            \fill (z\i4) circle[radius=2pt];

            \draw (z\i2) -- (z\i1) -- (z\i3) -- (z\i4) -- (z\i2);

        }
        \foreach \i [evaluate={\j=int(\i+1)}] in {0,1,2}
        {
            \draw (z\i2) -- (z\j2);
            \draw (z\i3) -- (z\j3);

        }
        \draw (z01) -- (z04);
        \draw (z32) -- (z33);

    \end{tikzpicture}
    \caption{The graph $F_{0,4}$}\label{f04}
    \end{minipage}

\end{figure}
Now, an interesting problem is whether $\frac{n}{r}$ is always the upper bound of $\gamma_{p,k}(G)$  when $G$ is claw-free. We will discuss this problem in next section.

\section{Claw-free regular graphs}
First, we establish the relation between $k$-power domination and domination by presenting Observation \ref{dp}. Then, we use Observation \ref{dp} to construct a series of regular claw-free graphs satisfying that $\gamma_{p,k}(G)=\frac{4n}{3(r+1)}>\frac{n}{r}$, where $r>3$.

A set $S$ of vertices in a graph $G$ is called a \emph{domination set} (abbreviated as DS) of $G$  if every vertex of $V\setminus S$ is adjacent to some vertex of $G$. The minimum cardinality of a DS of $G$ is the \emph{domination number} of $G$, denoted by $\gamma(G)$.

\begin{observation}\label{dp}
    For each $k\ge 1$ and $r\ge 1$, if $G$ is a connected $r$-regular claw-free graph of order $n$, then there exists a connected $r'$-regular claw-free graph $G'$ of order $n'=(k+1)n$ such that $r'=kr+r+k$ and $\gamma_{p,k}(G')=\gamma(G)$.
\end{observation}
\begin{proof}
    Let $V(G)=\{v_1,v_2,\cdots,v_n\}$. Let $G'$ be the graph constructed from $G$ as follows. Take $n$ disjoint cliques $V_i=\{v_i^1,v_i^2,\cdots,v_i^{k+1}\}$ corresponding to $v_i$. For each edge $v_iv_j\in E(G)$, add the edges $v_i^s v_j^q$ for each $s,q\in \{1,2,\cdots,k+1\}$ (see Figure \ref{fdp}). It is easy to check that $G'$ is a claw-free graph.

 %We claim that $G'$ is a claw-free graph. Let $u\in V(G')$ and $w_1,w_2,w_3\in N_{G'}(u)$ be three different vertices. If there exists $i$ such that $|V_i\cap\{w_1,w_2,w_3\}|\ge 2$, then $G'[w_1\cup w_2\cup w_3]$ have at least one edge. Now we assume that $|V_i\cap\{w_1,w_2,w_3\}|\le 1$ for each $i\in\{1,2,\cdots,n\}$. Since $G$ is a claw-free graph, $G'[w_1\cup w_2\cup w_3]$ have at least one edge. Thus $G'$ is a claw-free graph.

    Let $S=\{v_{i_1},v_{i_2},\cdots,v_{i_t}\}$ be a DS of $G$ with $t=\gamma(G)$. Then $\{v_{i_1}^1,v_{i_2}^1,\cdots,v_{i_t}^1\}$ is a $k$-PDS of $G'$, implying that $\gamma_{p,k}(G')\le \gamma(G)$. On the other hand, let $S'=\{v_{i_1}^{j_1},v_{i_2}^{j_2},\cdots,v_{i_t}^{j_t}\}$ be a $k$-PDS of $G'$ with $t=\gamma_{p,k}(G')$. If there exists $i\in\{1,2,\cdots,n\}$ such that $|S'\cap V_i|\ge 2$, then $S''=(S'\backslash V_i) \cup \{v_i^1\}$ is also a $k$-PDS of $G'$ with $|S''|<|S'|$, a contradiction. Hence, $|S'\cap V_i|\le 1$ for each $i\in\{1,2,\cdots,n\}$. Thus, $\{v_{i_1},v_{i_2},\cdots,v_{i_t}\}$ is a DS of $G'$, implying that $\gamma_{p,k}(G')\ge \gamma(G)$.
\end{proof}

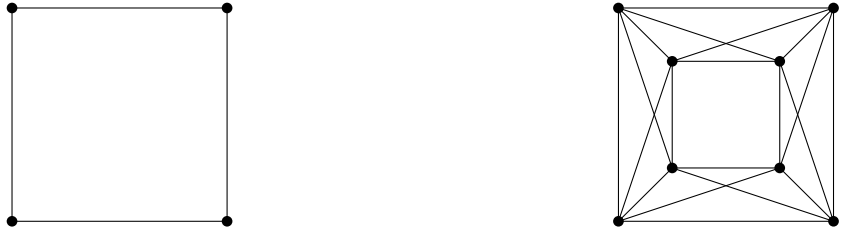
\begin{figure}[H]
    \centering
    \begin{minipage}[t]{0.49\textwidth}
    \centering
    \begin{tikzpicture}
        \foreach \i in {1,2,3,4}
        {
            \coordinate (a\i) at (-45+90*\i:2);
            \fill (a\i) circle[radius=2pt];
        }
        \draw (a1) -- (a2) -- (a3) -- (a4) -- cycle;
    \end{tikzpicture}
    \end{minipage}
    \begin{minipage}[t]{0.49\textwidth}
        \centering
        \begin{tikzpicture}
            \foreach \i in {1,2,3,4}
            {
                \foreach \j in {1,2}
                {
                    \coordinate (a\i\j) at (-45+90*\i:\j);
                    \fill (a\i\j) circle[radius=2pt];
                }
                \draw (a\i1) -- (a\i2);
            }
            \foreach \i [evaluate={\l=int(\i+1)}] in {1,2,3}
            {
                \foreach \j in {1,2}
                {
                    \foreach \k in {1,2}
                    {
                        \draw (a\i\j) -- (a\l\k);
                    }
                }
            }
            \foreach \j in {1,2}
            {
                \foreach \k in {1,2}
                {
                    \draw (a1\j) -- (a4\k);
                }
            }
        \end{tikzpicture}
        \end{minipage}
    \caption{An example of transformation in Observation \ref{dp} for $k=1$}\label{fdp}
\end{figure}

%\begin{observation}\label{Hkq}
%    For each $k\ge 1$ and $q\ge 1$, there exists a connected $(4k+3)$-regular claw-free graph $H_{k,q}$ of order $n=6(k+1)q$ such that $\gamma_{p,k}(H_{k,q})=2q=\frac{4}{3}\frac{n}{4k+4}$.
%\end{observation}
%\begin{proof}
    Let $H$ be the graph of order 6 as drawn in Figure \ref{h}. We define the graph $H_{0,q}$ as follows. Take $q$ disjoint copies $H_i\cong H$, where $i=1,2,\cdots,q$. For each $i\in\{1,2,\cdots,q\}$, let $x_i,y_i\in V(H_i)$ such that $d_{H_i}(x_i)=d_{H_i}(y_i)=2$. Add the edges $y_ix_{i+1}$, where $i=1,2,\cdots,q$ and $x_{q+1} = x_1$ (see Figure \ref{h04}). It is clear that $H_{0,q}$ is a connected $3$-regular claw-free graph of order $6q$. By Observation \ref{dp}, we can construct $H_{k,q}$ $(=G')$  from $H_{0,q}$ $(=G)$.

    Let $S=\bigcup_{i=1}^q \{x_i,y_i\}$. Then $S$ is a DS of $H_{0,q}$, implying that $\gamma(H_{0,q})\le 2q$. Since $\gamma(C_4)=2$, we get $\gamma(H_{0,q})\ge 2q$. So $\gamma(H_{0,q})=2q$. By Observation \ref{dp},  $\gamma_{p,k}(H_{k,q})=\gamma(H_{0,q})=2q=\frac{4}{3}\frac{n'}{4k+4}=\frac{4n'}{3(r'+1)}>\frac{n'}{r'}$. Hence, $\frac{n'}{r'}$ is not always the upper bound of $\gamma_{p,k}(G')$  when $G'$ is claw-free.
%\end{proof}

\begin{figure}[H]
    \centering
    \centering
    \begin{minipage}[t]{0.49\textwidth}
    \centering
    \begin{tikzpicture}
        \foreach \i in {2}
        {
            \coordinate[label=left:{$x$}] (x\i) at (60+90*\i:2*2);
            \coordinate[label=right:{$y$}] (y\i) at (120+90*\i:2*2);

            \coordinate (z\i1) at (75+90*\i:1.41*2);
            \coordinate (z\i2) at (80+90*\i:2.13*2);
            \coordinate (z\i3) at (100+90*\i:2.13*2);
            \coordinate (z\i4) at (105+90*\i:1.41*2);

            \fill (x\i) circle[radius=2pt];
            \fill (y\i) circle[radius=2pt];
            \fill (z\i1) circle[radius=2pt];
            \fill (z\i2) circle[radius=2pt];
            \fill (z\i3) circle[radius=2pt];
            \fill (z\i4) circle[radius=2pt];

            \draw (x\i) -- (z\i1) -- (z\i4) -- (y\i) -- (z\i3) -- (z\i2) -- (x\i);
            \draw (z\i1) -- (z\i2);
            \draw (z\i3) -- (z\i4);

            \fill (0,-6) circle[radius=0pt];
        }

    \end{tikzpicture}
    \caption{The graph $H$}\label{h}
    \end{minipage}
    \centering
    \begin{minipage}[t]{0.49\textwidth}
    \centering
    \begin{tikzpicture}
        \foreach \i in {1,...,4}
        {
            \coordinate (x\i) at (60+90*\i:2);
            \coordinate (y\i) at (120+90*\i:2);
            \node at (60+90*\i:2.27) {$x_\i$};
            \node at (120+90*\i:2.27) {$y_\i$};

            \coordinate (z\i1) at (75+90*\i:1.41);
            \coordinate (z\i2) at (80+90*\i:2.13);
            \coordinate (z\i3) at (100+90*\i:2.13);
            \coordinate (z\i4) at (105+90*\i:1.41);

            \fill (x\i) circle[radius=2pt];
            \fill (y\i) circle[radius=2pt];
            \fill (z\i1) circle[radius=2pt];
            \fill (z\i2) circle[radius=2pt];
            \fill (z\i3) circle[radius=2pt];
            \fill (z\i4) circle[radius=2pt];

            \draw (x\i) -- (z\i1) -- (z\i4) -- (y\i) -- (z\i3) -- (z\i2) -- (x\i);
            \draw (z\i1) -- (z\i2);
            \draw (z\i3) -- (z\i4);

        }
        \foreach \i[evaluate={\j=int(\i+1)}] in {1,...,3}
        {
            \draw (y\i) -- (x\j);
        }
        \draw (y4) -- (x1);

    \end{tikzpicture}
    \caption{The graph $H_{0,4}$}\label{h04}
    \end{minipage}

\end{figure}
Now we know that in Conjecture \ref{DorbConj}, if $r-k$ is sufficiently large, then $\frac{n}{r}$ is not always the upper bound of $\gamma_{p,k}(G)$. %So, we consider Conjecture \ref{DorbConj} and Question \ref{DorbQues} when $G$ is a connected claw-free regular graph.
%\begin{question}\label{CfreeQues}
%Add claw-free condition to the graph $G$ in Conjecture \ref{DorbConj}. Determine the smallest positive value $s_{min}(r)$ such that if $k\ge s_{min}(r)$, then $\gamma_{p,k}(G)\le \frac{n}{r+1}$.
%\end{question}
%\subsection{Regular claw-free graphs }
For each $r\ge 4$ and $k=\lfloor\frac{r}{2}\rfloor-1$, we will show that Conjecture \ref{DorbConj} does not always hold  for claw-free $r$-regular graphs  by presenting Observations \ref{odd} and \ref{even}. It means that $k_{min}(r)\ge \lfloor\frac{r}{2}\rfloor$ even restricted to claw-free regular graphs in the Question \ref{DorbQues}.%if $G$ is a connected claw-free $r$-regular graph,

\begin{observation}\label{odd}
    For each odd $r\ge 5$ and $q\ge 1$, there exists a connected claw-free $r$-regular graph $G_{r,q}$ of order $n=|V(G_{r,q})|$ such that $\gamma_{ p,\frac{r-3}{2}}(G_{r,q})=\frac{n+2}{r+1}>\frac{n}{r+1}$.
\end{observation}
\begin{proof}
    % We construct $G_{r,q}$ by the following steps. First, take a path of order $3q+2$ and label the vertices to $a_0,b_0,u_1,\dots,a_{q-1},b_{q-1},u_q,a_q,b_q$. Then split $a_i$ to clique $A_i=\{a_i^1,\dots, a_i^{(r-1)/2}\}$, split $b_i$ to clique $B_i=\{b_i^1,\dots, b_i^{(r-1)/2}\}$, and split $u_i$ to clique $U_i=\{u_i^1,u_i^2\}$. Note that the meaning of ``$u$ split to a clique $U$'' is  Finally, add the edges $a_0^jb_q^j$ and $a_0^jb_q^{j+1}$ for each  $j\in\{1,\cdots,\frac{r-1}{2}\}$, where $b_q^{\frac{r+1}{2}}=b_q^1$.

    We define $A_i=\{a_i^1,\cdots, a_i^{(r-1)/2}\}$, $B_i=\{b_i^1,\cdots, b_i^{(r-1)/2}\}$ and $U_i=\{u_i^1,u_i^2\}$ for each $i\in\{0,1,\cdots,q\}$. Then, we construct $G_{r,q}$ by the following steps. Firstly, let $V(G_{r,q})=(A_{0}\cup B_{0})\cup\left(\bigcup_{i=1}^q{(U_i\cup A_i\cup B_i)}\right)$. Secondly, add the edges such that $A_q\cup B_q$, $A_i\cup B_i$, $B_i\cup U_{i+1}$ and $U_{i+1}\cup A_{i+1}$ are cliques for each $i\in\{0,1,\cdots,q-1\}$. Finally, add the edges $a_0^jb_q^j$ and $a_0^jb_q^{j+1}$ for each  $j\in\{1,\cdots,\frac{r-1}{2}\}$, where $b_q^{\frac{r+1}{2}}=b_q^1$ (see Figures \ref{odd1}-\ref{odd3}).

    % Let $U_i=\{u_i^1,u_i^2\}$, $A_i=\{a_i^1,\cdots,a_i^{\frac{r-1}{2}}\}$, $B_i=\{b_i^1,\cdots,b_i^{\frac{r-1}{2}}\}$. Then we define the graph $G_{r,q}$ (see Figures 2-4) such that
    % \begin{align*}
    %     N_G[u_i^j]&=B_{i-1}\cup A_i\cup U_i && for~each~i\in\{1,\cdots,q\},~ j\in\{1,2\},\\
    %     N_G[a_i^j]&=A_i\cup B_i\cup U_i && for~each~i\in\{1,\cdots,q\},~ j\in\{1,\cdots,\frac{r-1}{2}\},\\
    %     N_G[b_i^j]&=A_i\cup B_i\cup U_{i+1}&& for~each~ i\in\{0,\cdots,q-1\},~ j\in\{1,\cdots,\frac{r-1}{2}\},\\
    %     N_G[a_0^j]&=A_0\cup B_0\cup \{b_q^j,b_q^{j+1}\} && for~each~ j\in\{1,\cdots,\frac{r-1}{2}\},\\
    %     N_G[b_q^j]&=A_q\cup B_q\cup \{a_0^j,a_0^{j-1}\} && for~each~ j\in\{1,\cdots,\frac{r-1}{2}\}.
    % \end{align*}
    % where $V=V(G_{r,q})=(A_{0}\cup B_{0})\cup\left(\bigcup_{i=1}^q{(U_i\cup A_i\cup B_i)}\right)$ and $a_0^0=a_0^{\frac{r-1}{2}},b_q^{\frac{r+1}{2}}=b_q^1$.

    % It is obvious that $|V|=(r-1)+q(r+1)=(q+1)(r+1)-2$. For each vertex $v\in V$, we have $|N_G[v]|=2\cdot \frac{r-1}{2}+2=r+1$. Thus,
    It is easy to check that $G_{r,q}$ is a connected $r$-regular claw-free graph of order $n=(q+1)(r+1)-2$. Let $k=\frac{r-3}{2}$.
    Since $\{a_0^1,\cdots,a_q^1\}$ is a $k$-PDS of $G_{r,q}$, we have $\gamma_{p,k}(G_{r,q})\le q+1$. On the other hand, let $S$ be a $k$-PDS of $G_{r,q}$. It is clear that $A_q$ is a $k$-fort and $B_i$ is also a $k$-fort for each $i\in \{0,\cdots,q-1\}$. By Propostion \ref{k-fort}, $|S\cap (A_q\cup B_q\cup U_{q})|\ge 1$ and $|S\cap (A_i\cup B_i\cup U_{i+1})|\ge 1$ for each $i\in \{0,\cdots,q-1\}$. It leads to $|S|\ge q$. Moreover, if $|S|=q$, then $|S\cap U_i|=1$ for each $i\in \{1,\cdots,q\}$. In this case, $P_{G_{r,q}}^\infty(S)=V\backslash (A_0\cup B_q)$, contradicting that $S$ is a $k$-PDS of $G_{r,q}$. Hence, $\gamma_{p,k}(G_{r,q})=q+1=\frac{n+2}{r+1}>\frac{n}{r+1}$.

    % \begin{align*}
    %     P_{G_{r,q}}^\infty(V\backslash (A_i\cup B_i\cup U_{i+1}))=V\backslash B_i &\Rightarrow S\cap (A_i\cup B_i\cup U_{i+1})\neq \emptyset && \forall i\in \{0,\dots,q-1\}\\
    %     P_{G_{r,q}}^\infty(V\backslash (A_i\cup B_i\cup U_i))=V\backslash A_i &\Rightarrow S\cap (A_i\cup B_i\cup U_i)\neq \emptyset &&\forall i\in \{1,\dots,q\}
    % \end{align*}
    % It is clear that $|S|\ge q$. But if $|S|=q$, then it is easy to prove that $|S\cap U_i|=1$ for each $i\in \{1,\cdots,q\}$ which is a contradiction since $P_{G_{r,q}}^\infty(\{U_1,\cdots,U_q\})=V\backslash (A_0\cup B_q)$. Thus $\gamma_{p,\frac{r-3}{2}}(G_{r,q})=q+1=\frac{n+2}{r+1}$.
\end{proof}
\begin{figure}[H]
    \centering
    \begin{minipage}[t]{0.32\textwidth}
    \centering
    \begin{tikzpicture}
        \foreach \i in {0,1}
        {
            \foreach \j in {1,2}
            {
                \coordinate (a\i\j) at (2*\i-0.5,2-\j);
                \node at (2*\i-0.5,2.9-1.6*\j) {$a_\i^\j$};
                \fill (a\i\j) circle[radius=2pt];
                \coordinate (b\i\j) at (2*\i-1.5,2-\j);
                \node at (2*\i-1.5,2.9-1.6*\j) {$b_\i^\j$};
                \fill (b\i\j) circle[radius=2pt];
            }
            \draw (a\i1) -- (a\i2);
            \draw (b\i1) -- (b\i2);
        }
        \foreach \i in {0,1}
        {
            \foreach \j in {1,2}
            {
                \foreach \k in {1,2}
                {
                    \draw (a\i\j) -- (b\i\k);
                }
            }
        }
        \foreach \i in {1}
        {
            \foreach \j in {1,2}
            {
                \coordinate (u\i\j) at (0,-0.9*\j);
                \node at (0,-1.5*\j+0.9) {$u_\i^\j$};
                \fill (u\i\j) circle[radius=2pt];
            }
            \draw (u\i1) -- (u\i2);
        }

        \foreach \k in {1,2}
        {
            \draw (a12) -- (u1\k);
            \draw (b02) -- (u1\k);
        }
        \draw (u11) .. controls (-2,-1) .. (b01);
        \draw (u12) .. controls (-2.3,-1) .. (b01);
        \draw (u11) .. controls (2,-1) .. (a11);
        \draw (u12) .. controls (2.3,-1) .. (a11);
        \foreach \j[evaluate={\k=int(\j+1)}] in {1}
        {
            \draw (a0\j) -- (b1\j);
            \draw (a0\j) -- (b1\k);
        }
        \draw (a02)--(b12);
        \draw (a02)--(b11);
    \end{tikzpicture}
    \caption{The graph $G_{5,1}$}\label{odd1}
    \end{minipage}
    \begin{minipage}[t]{0.32\textwidth}
        \centering
        \begin{tikzpicture}
            \foreach \i in {0,2}
            {
                \foreach \j in {1,2}
                {
                    \coordinate (a\i\j) at (\i-0.5,2-\j);
                    \node at (\i-0.5,2.9-1.6*\j) {$a_\i^\j$};
                    \fill (a\i\j) circle[radius=2pt];
                    \coordinate (b\i\j) at (\i-1.5,2-\j);
                    \node at (\i-1.5,2.9-1.6*\j) {$b_\i^\j$};
                    \fill (b\i\j) circle[radius=2pt];
                }
            }
            \foreach \j in {1,2}
            {
                \coordinate (a1\j) at (-0.4,-0.8*\j-0.2);
                \node at (-0.4,-1.4*\j+0.7) {$a_1^\j$};
                \fill (a1\j) circle[radius=2pt];
                \coordinate (b1\j) at (0.4,-0.8*\j-0.2);
                \node at (0.4,-1.4*\j+0.7) {$b_1^\j$};
                \fill (b1\j) circle[radius=2pt];

            }
            \foreach \i in {0,1,2}
            {
                \foreach \j in {1,2}
                {
                    \foreach \k in {1,2}
                    {
                        \draw (a\i\j) -- (b\i\k);
                    }
                }
                \draw (a\i1) -- (a\i2);
                \draw (b\i1) -- (b\i2);
            }
            \foreach \i in {1,2}
            {
                \foreach \j in {1,2}
                {
                    \coordinate (u\i\j) at (2.4*\i-3.6,-0.8*\j-0.2);
                    \node at (2.4*\i-3.6,-1.4*\j+0.7) {$u_\i^\j$};
                    \fill (u\i\j) circle[radius=2pt];
                }
                \draw (u\i1) -- (u\i2);
            }

            \foreach \k in {1,2}
            {
                \draw (a22) -- (u2\k);
                \draw (b02) -- (u1\k);
                \foreach \i in {1,2}
                {
                    \draw (a1\i) -- (u1\k);
                    \draw (b1\i) -- (u2\k);
                }
            }
            \draw (u11) .. controls (-2,-1) .. (b01);
            \draw (u12) .. controls (-2.3,-1) .. (b01);
            \draw (u21) .. controls (2,-1) .. (a21);
            \draw (u22) .. controls (2.3,-1) .. (a21);
            \foreach \j[evaluate={\k=int(\j+1)}] in {1}
            {
                \draw (a0\j) -- (b2\j);
                \draw (a0\j) -- (b2\k);
            }
            \draw (a02)--(b22);
            \draw (a02)--(b21);
        \end{tikzpicture}
        \caption{The graph $G_{5,2}$}
    \end{minipage}
    \begin{minipage}[t]{0.33\textwidth}
        \centering
        \begin{tikzpicture}
            \foreach \i in {0,2}
            {
                \foreach \j in {1,2}
                {
                    \coordinate (a\i\j) at (\i-0.5,2-\j);
                    \fill (a\i\j) circle[radius=2pt];
                    \coordinate (b\i\j) at (\i-1.5,2-\j);
                    \fill (b\i\j) circle[radius=2pt];
                }
            }

            \foreach \j in {1,2}
            {
                \coordinate (a1\j) at (-0.4,-0.8*\j-0.2);
                \node at (-0.4,-1.4*\j+0.7) {$a_1^\j$};
                \fill (a1\j) circle[radius=2pt];
                \coordinate (b1\j) at (0.4,-0.8*\j-0.2);
                \node at (0.4,-1.4*\j+0.7) {$b_{q}^\j$};
                \fill (b1\j) circle[radius=2pt];

            }
            \foreach \i in {0,1,2}
            {
                \draw (a\i1) -- (a\i2);
                \draw (b\i1) -- (b\i2);
            }
            \foreach \i in {0,2}
            {
                \foreach \j in {1,2}
                {
                    \foreach \k in {1,2}
                    {
                        \draw (a\i\j) -- (b\i\k);
                    }
                }
            }
            % \draw[dotted] (-0.4,-1.4) -- (0.4,-1.4);
            \node at (0,-1.4) {$\dots$};
            \foreach \i in {1,2}
            {
                \foreach \j in {1,2}
                {
                    \coordinate (u\i\j) at (2.4*\i-3.6,-0.8*\j-0.2);

                    \fill (u\i\j) circle[radius=2pt];
                }
                \draw (u\i1) -- (u\i2);
            }
            \foreach \j in {1,2}
            {
                \node at (-0.5,2.9-1.6*\j) {$a_0^\j$};
                \node at (1.5,2.9-1.6*\j) {$a_{q+1}^\j$};
                \node at (-1.5,2.9-1.6*\j) {$b_0^\j$};
                \node at (0.5,2.9-1.6*\j) {$b_{q+1}^\j$};
                \node at (-1.2,-1.4*\j+0.7) {$u_1^\j$};
                \node at (1.2,-1.4*\j+0.7) {$u_{q+1}^\j$};
            }
            \foreach \k in {1,2}
            {
                \draw (a22) -- (u2\k);
                \draw (b02) -- (u1\k);
                \foreach \i in {1,2}
                {
                    \draw (a1\i) -- (u1\k);
                    \draw (b1\i) -- (u2\k);
                }
            }
            \draw (u11) .. controls (-2,-1) .. (b01);
            \draw (u12) .. controls (-2.3,-1) .. (b01);
            \draw (u21) .. controls (2,-1) .. (a21);
            \draw (u22) .. controls (2.3,-1) .. (a21);
            \foreach \j[evaluate={\k=int(\j+1)}] in {1}
            {
                \draw (a0\j) -- (b2\j);
                \draw (a0\j) -- (b2\k);
            }
            \draw (a02)--(b22);
            \draw (a02)--(b21);
        \end{tikzpicture}
        \caption{The graph $G_{5,q+1}$}\label{odd3}
    \end{minipage}
\end{figure}
\begin{observation}\label{even}
     For each even $r\ge 4$ and $q\ge 1$, there exists a connected claw-free $r$-regular graph $G_{r,q}$ of order $n=|V(G_{r,q})|$ such that $\gamma_{ p,\frac{r-2}{2}}(G_{r,q})=\frac{n+1}{r+1}>\frac{n}{r+1}$.
\end{observation}

\begin{proof}

% We consider a graph $G_{r,q}$ which was presented by Lu et al. \cite{Lu2018} and was noted by $Q_{r,k}$ in their paper. It is easy to check that $G_{r,q}$ is a connected claw-free $r$-regular graph of order $(q+1)(r+1)-1$. Let $U_i=\{u_i\}$, $A_i=\{a_i^1,\cdots,a_i^{\frac{r}{2}}\}$, $B_i=\{b_i^1,\cdots,b_i^{\frac{r}{2}}\}$. Then we redefine the graph $G_{r,q}$ (see Figures 5-7) such that
%     \begin{align*}
%         N_G[u_i]&=B_{i-1}\cup A_i\cup U_i && for~each~ i\in\{1,\cdots,q\},\\
%         N_G[a_i^j]&=A_i\cup B_i\cup U_i && for~each~ i\in\{1,\cdots,q\},~ j\in\{1,\cdots,\frac{r}{2}\},\\
%         N_G[b_i^j]&=A_i\cup B_i\cup U_{i+1}&& for~each~ i\in\{0,\cdots,q-1\},~j\in\{1,\cdots,\frac{r}{2}\},\\
%         N_G[a_0^j]&=A_0\cup B_0\cup \{b_q^j\} && for~each~ j\in\{1,\cdots,\frac{r}{2}\},\\
%         N_G[b_q^j]&=A_q\cup B_q\cup \{a_0^j\} && for~each~ j\in\{1,\cdots,\frac{r}{2}\}.
%     \end{align*}
%     where $V=V(G_{r,q})=(A_{0}\cup B_{0})\cup\left(\bigcup_{i=1}^q{(U_i\cup A_i\cup B_i)}\right)$.

    % It is obvious that $|V|=(r-1)+q(r+1)=(q+1)(r+1)-2$. For each vertex $v\in V$, we have $|N_G[v]|=2\cdot \frac{r-1}{2}+2=r+1$. Thus,
    We consider a graph $G_{r,q}$ which was presented by Lu et al. in \cite{Lu2018} and was noted by $Q_{r,k}$ in their paper. Let $A_i=\{a_i^1,\cdots, a_i^{r/2}\}$, $B_i=\{b_i^1,\cdots, b_i^{r/2}\}$ and $U_i=\{u_i\}$ for each $i\in\{0,1,\cdots,q\}$. Now we redefine $G_{r,q}$ by the following steps. Firstly, let $V(G_{r,q})=(A_{0}\cup B_{0})\cup\left(\bigcup_{i=1}^q{(U_i\cup A_i\cup B_i)}\right)$. Secondly, add the edges such that $A_q\cup B_q$, $A_i\cup B_i$, $B_i\cup U_{i+1}$ and $U_{i+1}\cup A_{i+1}$ are cliques for each $i\in\{0,\cdots,q-1\}$. Finally, add the edges $a_0^jb_q^j$ for each  $j\in\{1,\cdots,\frac{r}{2}\}$ (see Figures \ref{even1}-\ref{even3}).

    It is easy to check that $G_{r,q}$ is a connected claw-free $r$-regular graph. Similar to the proof of Observation \ref{odd}, we have $\gamma_{p,\frac{r-2}{2}}(G_{r,q})=q+1=\frac{n+1}{r+1}>\frac{n}{r+1}$.
    % We have $\gamma_{P,\frac{r-2}{2}}(G_{r,q})\le q+1$ since $N_G[\{a_0^1,\dots,a_q^1\}]=V$. On the other hand, if $S$ is a power domination set of $G_{r,q}$, then we have
    % \begin{align*}
    %     P_{G_{r,q}}^\infty(V\backslash N_G[B_i])=V\backslash B_i &\Rightarrow S\cap N_G[B_i]\neq \emptyset && \forall i\in \{0,\dots,q-1\}\\
    %     P_{G_{r,q}}^\infty(V\backslash N_G[A_i])=V\backslash A_i &\Rightarrow S\cap N_G[A_i]\neq \emptyset &&\forall i\in \{1,\dots,q\}
    % \end{align*}
    % It's clear that $|S|\ge q$. But if $|S|=q$, then it's easy to prove that $|S\cap U_i|=1$ for each $i\in \{1,\dots,q\}$ which is a contradiction since $P_{G_{r,q}}^\infty(\{U_1,\dots,U_p\})=V\backslash (A_0\cup B_q)$. Thus $\gamma_{P,\frac{r-2}{2}}(G_{r,q})=q+1=\frac{n+2}{r+1}$.
\end{proof}
\begin{figure}[H]
    \centering
    \begin{minipage}[t]{0.32\textwidth}
    \centering
    \begin{tikzpicture}
        \foreach \i in {0,1}
        {
            \foreach \j in {1,2}
            {
                \coordinate (a\i\j) at (2*\i-0.5,2-\j);
                \node at (2*\i-0.5,2.9-1.6*\j) {$a_\i^\j$};
                \fill (a\i\j) circle[radius=2pt];
                \coordinate (b\i\j) at (2*\i-1.5,2-\j);
                \node at (2*\i-1.5,2.9-1.6*\j) {$b_\i^\j$};
                \fill (b\i\j) circle[radius=2pt];
            }
            \draw (a\i1) -- (a\i2);
            \draw (b\i1) -- (b\i2);
        }
        \foreach \i in {0,1}
        {
            \foreach \j in {1,2}
            {
                \foreach \k in {1,2}
                {
                    \draw (a\i\j) -- (b\i\k);
                }
            }
        }
        \coordinate (u11) at (0,-1.4);
        \node at (0,-1.7) {$u_1$};
        \fill (u11) circle[radius=2pt];
        \draw (a12) -- (u11);
        \draw (b02) -- (u11);
        \draw (u11) .. controls (-2,-1) .. (b01);
        \draw (u11) .. controls (2,-1) .. (a11);
        \draw (a01) -- (b21);
        \draw (a02)--(b22);
        \node at (0,-2.3) { };
    \end{tikzpicture}
    \caption{The graph $G_{4,1}$}\label{even1}
    \end{minipage}
    \begin{minipage}[t]{0.32\textwidth}
        \centering
        \begin{tikzpicture}
            \foreach \i in {0,2}
            {
                \foreach \j in {1,2}
                {
                    \coordinate (a\i\j) at (\i-0.5,2-\j);
                    \node at (\i-0.5,2.9-1.6*\j) {$a_\i^\j$};
                    \fill (a\i\j) circle[radius=2pt];
                    \coordinate (b\i\j) at (\i-1.5,2-\j);
                    \node at (\i-1.5,2.9-1.6*\j) {$b_\i^\j$};
                    \fill (b\i\j) circle[radius=2pt];
                }
            }
            \foreach \j in {1,2}
            {
                \coordinate (a1\j) at (-0.4,-0.8*\j-0.2);
                \node at (-0.4,-1.4*\j+0.7) {$a_1^\j$};
                \fill (a1\j) circle[radius=2pt];
                \coordinate (b1\j) at (0.4,-0.8*\j-0.2);
                \node at (0.4,-1.4*\j+0.7) {$b_1^\j$};
                \fill (b1\j) circle[radius=2pt];

            }
            \foreach \i in {0,1,2}
            {
                \foreach \j in {1,2}
                {
                    \foreach \k in {1,2}
                    {
                        \draw (a\i\j) -- (b\i\k);
                    }
                }
                \draw (a\i1) -- (a\i2);
                \draw (b\i1) -- (b\i2);
            }
            \foreach \i in {1,2}
            {
                \coordinate (u\i1) at (2.4*\i-3.6,-1.4);
                \node at (2.4*\i-3.6,-1.7) {$u_\i$};
                \fill (u\i1) circle[radius=2pt];
            }

            \foreach \k in {1}
            {
                \draw (a22) -- (u2\k);
                \draw (b02) -- (u1\k);
                \foreach \i in {1,2}
                {
                    \draw (a1\i) -- (u1\k);
                    \draw (b1\i) -- (u2\k);
                }
            }
            \draw (u11) .. controls (-2,-1) .. (b01);
            % \draw (u12) .. controls (-2.3,-1) .. (b01);
            \draw (u21) .. controls (2,-1) .. (a21);
            % \draw (u22) .. controls (2.3,-1) .. (a21);
            \draw (a01) -- (b21);
            \draw (a02)--(b22);

        \end{tikzpicture}
        \caption{The graph $G_{4,2}$}
    \end{minipage}
    \begin{minipage}[t]{0.33\textwidth}
        \centering
        \begin{tikzpicture}
            \foreach \i in {0,4}%定义顶部的点
            {
                \foreach \j in {1,2}
                {
                    \coordinate (a\i\j) at (\i/2-0.5,2-\j);
                    \fill (a\i\j) circle[radius=2pt];
                    \coordinate (b\i\j) at (\i/2-1.5,2-\j);
                    \fill (b\i\j) circle[radius=2pt];
                }
            }
            \foreach \i in {1,2}%定义下面左边的点
            {
                \foreach \j in {1,2}
                {
                    \coordinate (a\i\j) at (-2+0.22*1.73+1.19*\i-1.19,-0.8*\j-0.2);
                    \node at (-2+0.22*1.73+1.19*\i-1.19,-1.4*\j+0.7) {$a_\i^\j$};
                    \fill (a\i\j) circle[radius=2pt];
                    \coordinate (b\i\j) at (-2+0.22*1.73+0.44+1.19*\i-1.19,-0.8*\j-0.2);
                    \node at (-2+0.22*1.73+0.44+1.19*\i-1.19,-1.4*\j+0.7) {$b_\i^\j$};
                    \fill (b\i\j) circle[radius=2pt];
                }
            }
            \foreach \j in {1,2}%定义下面右边的点
            {
                \coordinate (a3\j) at (-2+0.22*1.73+1.19*2+0.44,-0.8*\j-0.2);
                \node at (-2+0.22*1.73+1.19*2+0.44,-1.4*\j+0.7) {$a_{q}^\j$};
                \fill (a3\j) circle[radius=2pt];
                \coordinate (b3\j) at (-2+0.22*1.73+0.44+1.19*2+0.44,-0.8*\j-0.2);
                \node at (-2+0.22*1.73+0.44+1.19*2+0.44,-1.4*\j+0.7) { $b_{q}^\j$};
                \fill (b3\j) circle[radius=2pt];
            }
            \foreach \i in {0,...,4}
            {
                \draw (a\i1) -- (a\i2);
                \draw (b\i1) -- (b\i2);
            }
            \foreach \i in {0,...,4}
            {
                \foreach \j in {1,2}
                {
                    \foreach \k in {1,2}
                    {
                        \draw (a\i\j) -- (b\i\k);
                    }
                }
            }
            % % \draw[dotted] (-0.4,-1.4) -- (0.4,-1.4);
            % \node at (0,-1.4) {$\dots$};
            \foreach \i in {1,2,3}%定义左边的u
            {
                \coordinate (u\i1) at (-2+1.19*\i-1.19,-1.4);
                \node at (-2+1.19*\i-1.19,-1.7) {$u_\i$};
                \fill (u\i1) circle[radius=2pt];
            }
            \foreach \i in {4,5}%定义右边的u
            {
                \coordinate (u\i1) at (2+1.19*\i-1.19*5,-1.4);
                \fill (u\i1) circle[radius=2pt];
            }
            \draw[dotted] (u31) -- (u41);
            \node at (2-1.19+0.05,-1.7) {$u_q$};
            \node at (2,-1.7) {$u_{q+1}$};
            \foreach \j in {1,2}
            {
                \node at (-0.5,2.9-1.6*\j) {$a_0^\j$};
                \node at (1.5,2.9-1.6*\j) {$a_{q+1}^\j$};
                \node at (-1.5,2.9-1.6*\j) {$b_0^\j$};
                \node at (0.5,2.9-1.6*\j) {$b_{q+1}^\j$};
            }
            \foreach \k in {1,2}
            {
                \foreach \i in {1,2}
                {
                    \draw (a\i\k) -- (u\i1);
                }
                \foreach[evaluate={\j=int(\i-1)}] \i in {1,2,3}
                {
                    \draw (b\j\k) -- (u\i1);
                }
                \draw (b3\k) -- (u51);
                \draw (a4\k) -- (u51);
                \draw (a3\k) -- (u41);
            }

            \draw (a01) -- (b41);
            \draw (a02)--(b42);
        \end{tikzpicture}
        \caption{The graph $G_{4,q+1}$}\label{even3}
    \end{minipage}
\end{figure}
Hence, we will consider  Conjecture \ref{DorbConj} when $G$ is a connected claw-free $r$-regular graph and $k\ge\lfloor\frac{r}{2}\rfloor$. It means that $k\ge \frac{r-1}{2}$. If we let $r=k+l+1$, we have $k\ge\frac{k+l}{2}$, implying that $k\ge l$. Chang et al. \cite{chang2012generalized} studied the case that $l=1$. We further studied the cases $l=2$ and $l=3$  by proving Theorem \ref{MainResult}.
%\section{The proof of Theorem \ref{MainResult}}

If the statement of Theorem \ref{MainResult} fails, then we suppose that $G$ is a counterexample with minimal $|V(G)|$, i.e, $G$ is a connected claw-free $(k+l+1)$-regular graph of minimal order $n$ and $\gamma_{p,k}(G)>\frac{n}{k+l+2}$ for $l\in\{2,3\}$ and $k\ge l$.

Before giving the proof of Theorem \ref{MainResult}, we define an important structure, which is an $L$-configuration in $G$.
\begin{definition}($L$-configuration).
    The subgraph $H\cong G[N[L]]$ is an $L$-configuration if $L$ is both a clique and a $k$-fort of $G$.

\end{definition}

Let $j\le k$ be a positive integer and $A_j$ be  the graph obtained from $K_{k+j+2}$ by removing $j$ edges which share a common vertex in $K_{k+j+2}$ (see Figures  \ref{a2}-\ref{a3}). Remark that $A_j$ is an $L$-configuration in $G$.
\vskip 1em

\begin{figure}[H]
    \begin{minipage}[t]{0.49\textwidth}
    \centering
    \begin{tikzpicture}
        \foreach \i in {1,...,5}
        {
            \coordinate (a\i) at (72*\i-54:1);
        }
        \coordinate (a6) at (0,2);
        \foreach \i in {1,...,6}
        {
            \fill (a\i) circle[radius=2pt];
        }
        \foreach \i in {1,2,3}
        {
            \draw (a6) -- (a\i);
        }
        \foreach \i in {1,...,4}
        {
            \foreach \j in {\i,...,5}
            {
                \draw (a\i) -- (a\j);
            }
        }
 \end{tikzpicture}
    \caption{$A_2$ for $k=2$}\label{a2}
    \end{minipage}
    \begin{minipage}[t]{0.49\textwidth}
    \centering
    \begin{tikzpicture}
        \foreach \i in {1,...,7}
        {
            \coordinate (a\i) at (360/7*\i+12.86:1);
        }
        \coordinate (a8) at (0,2);
        \foreach \i in {1,...,8}
        {
            \fill (a\i) circle[radius=2pt];
        }
        \foreach \i in {1,2,3,7}
        {
            \draw (a8) -- (a\i);
        }
        \foreach \i in {1,...,6}
        {
            \foreach \j in {\i,...,7}
            {
                \draw (a\i) -- (a\j);
            }
        }

    \end{tikzpicture}
    \caption{$A_3$ for $k=3$}\label{a3}
    \end{minipage}
\end{figure}

%\vskip 1em

Then, we present three useful lemmas.
\begin{lemma}\label{p5}
    Let $H$ be an $L$-configuration of $G$. If $S\subseteq L$ and $|S|\ge |L|-k$, then $N[S]=V(H)$. %and $F$ be subsets of vertices in $G$ that induce an $(L,F)$-configuration. Then $N[S]=L\cup F$ for each $S\subseteq L$ with $|S|\ge |L|-k$.
\end{lemma}
\begin{proof}
    Suppose that $S\subseteq L$ and $|S|\ge |L|-k$. It is clear that $L\subseteq N[S]\subseteq V(H)$. For each $v\in V(H)\setminus L$, we have $|N_L(v)\cap S|\ge 1$ since $L$ is a $k$-fort of $G$ and $|L|-|S|\le k$.  Hence, $v\in N[S]$, implying that $V(H)\subseteq N[S]$.
\end{proof}
\begin{lemma}\label{lf}
    Let $H$ be an $L$-configuration of $G$ and $H'$ be an $L'$-configuration of $G$. If $V(H)\cap V(H')\neq \emptyset$, then $V(H)=V(H')$.%$N[L]=N[L']$.
\end{lemma}
\begin{proof}
    For each $u\in V(H)\cap V(H')$, we define $S_u=N[u]\cap (L \cap L')$. Then $|S_u|=|N[u]\cap L|+|N[u]\cap L'|-|N[u]\cap (L \cup L')| $ according to the inclusion and exclusion principle.

    It is clear that $|L|-|N[u]\cap (L \cup L')|\ge (k+1)-(k+l+2)\ge -k-1$. We claim that the equation can't hold. Otherwise, suppose the equation holds. Then, we have $|L|=k+1$ and $N[u]\subseteq L \cup L'$. Without loss of generality, assume $u\in L$, and so $N[u]\backslash L\subseteq N[L]\setminus L$.  Since $L$ is a $k$-fort, $N(v)\cap L=L$ for each $v\in N[u]\backslash L$. Since $L'$ is a clique and $N[u]\setminus L\subseteq L'$, we have $N[u]\setminus L$ is a clique. It means that $N[u]$ is a clique, and so $G\cong K_{k+l+2}$, contradicting that $G$ is a counterexample. So, $|L|-|N[u]\cap (L \cup L')|\ge -k$.

  We claim that $L\cap L'\neq \emptyset$. Otherwise, suppose that $L\cap L'=\emptyset$. If $u\notin L \cup L'$ for each $u\in V(H)\cap V(H')$, then $d_G(u)\ge |L|+|L'|\ge 2(k+1)>k+l+1$, a contradiction.  Without loss of generality, we assume $u\in L$. Then $|S_u|=|N[u]\cap L'|+|L|-|N[u]\cap(L\cup L')|\ge |N[u]\cap L'|-k\ge 1$. It means that $|L\cap L'|\ge 1$, a contradiction. Hence, $L\cap L'\neq \emptyset$.

    Let $v\in L\cap L'$. Then $|S_v|= |L|+|L'|-|N[v]\cap (L \cup L')|$. It means that $|S_v|\ge |L|-k$ and $|S_v|\ge |L'|-k$. By Lemma \ref{p5}, $V(H)=N[S_v]=V(H')$.
\end{proof}
\begin{lemma}\label{lf2}
    Let $H$ be an $L$-configuration of $G$. Then, we have $V(H)\subseteq P^\infty(u)$ for each $u\in L$.
\end{lemma}
\begin{proof}
    Let $u\in L$. If $|L|= k+1$, then $N[u]=V(H)$ by Lemma \ref{p5}, implying that $V(H)\subseteq P^\infty(u)$. Now suppose that  $|L|\ge k+2$. Since $G$ is a $(k+l+1)$-regular graph and $l\le k$, $V(H)\subseteq P^\infty(u)$.
\end{proof}
We give the following method to choose a vertex subset $\mathcal{P}_0$ for $G$. First, let $\mathcal{P}_0=\emptyset$. Then, we process the following step. If $G$ contains an $L$-configuration and none vertex of $L$ is contained in $ P^\infty(\mathcal{P}_0)$, then we add one vertex of $L$ to $\mathcal{P}_0$. Process the step till $G$ contains no such an $L$-configuration.

By Lemmas \ref{lf} and \ref{lf2}, it is clear that $\mathcal{P}_0$ is a packing of $G$. We extend the packing $\mathcal{P}_0$ of $G$ to a maximal packing and denote the resulting packing by $S_0$.

\begin{lemma}\label{mainlma}
For $l\in\{2,3\}$ and $k\ge l$, $G$ has a sequence $S_0,S_1,\cdots,S_q$ such that the following holds:

(a) For all $t\ge 0$, $|S_{t+1}|=|S_t|+1$ and $|P^\infty(S_{t+1})|\geq|P^\infty(S_t)|+k+l+2$.

(b) $P^\infty(S_q)=V(G)$.
\end{lemma}
\begin{proof}
We prove part (a) and part (b) by induction on $t$. If $P^{\infty}(S_0)=V(G)$, then there is nothing to prove. Hence, we may assume that $P^{\infty}(S_0)\neq V(G)$. Let $t\ge 0$ and suppose that $S_{t}$ exists and $P^{\infty}(S_t)\neq V(G)$. Denote $M=P^{\infty}(S_t)$ and $\overline{M}=V(G) \setminus M$. Let $\mathcal{U}=\{u~|~u\in M$ and $N_G(u)\cap \overline{M}\neq \emptyset\}$. For each vertex $u\in \mathcal{U}$, since $N_G[u]\not\subseteq M$, we note that $d_{M}(u)\geq 1$ and $k+1\le d_{\overline{M}}(u)\le k+l$. Moreover, for each $u\in \mathcal{U}$, we define $L_u=N_G(u)\cap \overline{M}=\{u_1,u_2,\dots,u_{d_{\overline{M}}(u)}\}$, $F_u=N_G(L_u)\backslash L_u$ and $F'_u=F_u\setminus \{u\}$. Hence, $k+1\le|L_u|\le k+l$.

We claim that for each vertex $x\in\overline{M}$, $N_G(x)\cap \mathcal{U}\neq \emptyset$. Otherwise, suppose to the contrary that there exists $y\in\overline{M}$ such that $N_G(y)\cap\mathcal{U}=\emptyset$. Then $S_0\cup\{y\}$ is also a packing, contradicting that $S_0$ is a maximal packing. Now we present seven useful claims.

%If there is a vertex $x\in\overline M$ such that $N_G(x)\cap \mathcal{U}= \emptyset$, then $d_{\overline M}(x)=k+l+1$. We define $S_{t+1}=S_t\cup \{x\}$ and we let $j$ be the minimum integer such that  $P^j(S_t)=P^\infty(S_t)$. Then, $N[x]\subseteq P^0(S_{t+1})\subseteq P^j(S_{t+1})$. Since $N[x]\cap P^j(S_t)=\emptyset$, we have \[|P^\infty(S_{t+1})|\ge |P^\infty(S_t)|+|N[x]|\ge |P^\infty(S_t)|+k+l+2.\] So, part (a) follows as desired. Now we suppose that for each vertex $x\in\overline{M}$, $N_G(x)\cap \mathcal{U}\neq \emptyset$.
\begin{claim}\label{p1}
    If $H$ is an $L$-configuration of $G$, then $V(H)\subseteq M$. %all vertices of $H$ are contained in $M$. %\qed
\end{claim}
\begin{proof}
By the choose of $S_0$ and Lemma \ref{lf2}, we immediately obtain the Claim 1.
\end{proof}

\begin{claim}\label{c}
    For each $u\in \mathcal{U}$, $L_u$ induces a clique in $G$.
\end{claim}
\begin{proof}
Suppose $x_1$ and $x_2$ are two neighbors of $u$ in $L_u$ and $u$ is observed by $v$ in $M$. Then $x_1v,x_2v\notin E(G)$. If $x_1x_2\notin E(G)$, then $\{u,x_1,x_2, v\}$ induces a claw, a contradiction. Therefore, $L_u$ induces a clique in $G$.
\end{proof}

\begin{claim}\label{g1}
    Let $u\in\mathcal{U}$. If $|L_u|+|F_u\cap \overline{M}|\ge k+l+2$, then for $S_{t+1}=S_t\cup \{u_1\}$, we have
%\begin{center}
$|P^\infty(S_{t+1})|\geq |P^\infty(S_t)|+k+l+2$.
%\end{center}
\end{claim}
\begin{proof}
    Suppose $|L_u|+|F_u\cap \overline{M}|\ge k+l+2$. By Claim  \ref{c}, $L_u$ induces a clique in $G$. We define $S_{t+1}=S_t\cup \{u_1\}$ and we let $j$ be the minimum integer such that $P^j(S_t)=P^\infty(S_t)$. Then, $N[u_1]\subseteq P^0(S_{t+1})\subseteq P^j(S_{t+1})$, and so $L_u\cup\{u\}\subseteq P^j(S_{t+1})$. For each $u'\in L_u\setminus\{u_1\}$, we have
 \[|N(u')\setminus P^j(S_{t+1})|\le k+l+1-|L_u\setminus{u'}|-|\{u\}|\le l\le k.\]
\noindent It means that $N[u']\subseteq P^{j+1}(S_{t+1})$. Therefore, \[|P^\infty(S_{t+1})|\geq |P^\infty(S_t)|+|L_u|+|F_u\cap \overline{M}|\ge |P^\infty(S_t)|+k+l+2.\]
\end{proof}
\begin{claim}\label{g2}
    Let $u\in\mathcal{U}$. If there exists a vertex $w\in F_u\cap \overline{M}$ such that $|L_u|-d_{L_u}(w)\le k$ and $vw\notin E$ for each $v\in M\cap F_u$, then for $S_{t+1}=S_t\cup \{w\}$, we have
$|P^\infty(S_{t+1})|\geq |P^\infty(S_t)|+k+l+2$.
\end{claim}
\begin{proof}
    Suppose there exists a vertex $w\in F_u\cap \overline{M}$ such that $|L_u|-d_{L_u}(w)\le k$ and $vw\notin E$ for each $v\in M\cap F_u$. By Claim  \ref{c}, $L_u$ induces a clique in $G$. Since $N_G(w)\cap \mathcal{U}\neq \emptyset$, there exists $x\in \mathcal{U}$ such that $w\in L_x$. We claim that $L_x\cap L_u=\emptyset$. Otherwise, without loss of generality, assume $u_1\in L_x\cap L_u$. Then, $u_1x\in E$, and so $x\in F_u\cap M$. It leads to $xw\notin E$, a contradiction. Hence, $L_x\cap L_u=\emptyset$. We define $S_{t+1}=S_t\cup \{w\}$ and we let $j$ be the minimum integer such that $P^j(S_t)=P^\infty(S_t)$. Then, $N[w]\subseteq P^0(S_{t+1})\subseteq P^j(S_{t+1})$. By Claim \ref{c}, $L_x\subseteq P^j(S_{t+1})\setminus P^j(S_t)$. Since $|L_u|-d_{L_u}(w)\le k$, we have $L_u\subseteq P^{j+1}(S_{t+1})$. Therefore, we obtain
\[|P^\infty(S_{t+1})|\geq |P^\infty(S_t)|+|L_x|+|L_u|\ge |P^\infty(S_t)|+2(k+1)\ge |P^\infty(S_t)|+k+l+2.\]
\end{proof}

\begin{claim}\label{m1}
    If there is a vertex $u\in \mathcal{U}$ such that $|L_u|=k+l$, part (a) follows as desired.
\end{claim}
\begin{proof}
    Suppose there is a vertex $u\in \mathcal{U}$ such that $|L_u|=k+l$. By Claim  \ref{c}, $L_u$ induces a clique in $G$. If there is a vertex $w\in F'_u$ such that $d_{L_u}(w)\ge k+1$, then $G[\{u,w\}\cup L_u]$ is an $L$-configuration where $L=N_G(w)\cap L_u$, contradicting Claim \ref{p1}.

    Now we assume that $d_{L_u}(w)\le k$ for each $w\in F'_u$. Then, $|F'_u|\ge 2$. If there is a vertex $w\in F'_u$ such that $w\in M$, without loss of generality, suppose $u_1\in L_w$. Since $|L_w|\ge k+1$ and $d_{L_u}(w)\le k$, there is a vertex $w'\in L_w\setminus L_u$. By Claim \ref{c}, $u_1w'\in E$. It leads to $d(u_1)\ge |L_u\setminus\{u_1\}|+|\{u,w,w'\}|\ge k+l+2$, a contradiction. Now suppose $F'_u\subseteq \overline M$. Then, $|L_u|+|F_u\cap\overline M|=|L_u|+|F'_u|\ge k+l+2$. By Claim \ref{g1}, part (a) follows as desired.
\end{proof}
\begin{claim}\label{m2}
    When $l=3$, if $|L_u|=k+2$ for each $u\in \mathcal{U}$, part (a) follows as desired.
\end{claim}
\begin{proof}
    When $l=3$, suppose $|L_u|=k+2$ for each $u\in \mathcal{U}$. By Claim  \ref{c}, $L_u$ induces a clique in $G$. Since $G$ is a connected claw-free $(k+l+1)$-regular graph, $|N(u_1)\setminus(L_u\cup\{u\})|=k+l+1-(k+2)=2$, implying that $|F'_u|\ge 2$. We claim that $|F'_u|\ge 3$. Otherwise, we suppose $F'_u=\{w_1,w_2\}$, implying that $d_{L_u}(w_1)=d_{L_u}(w_2)=k+2$. Then, $G[L_u\cup F_u]$ is an $L$-configuration where $L=L_u$, contradicting Claim \ref{p1}. Hence, $|F'_u|\ge 3$. If $F'_u\cap M=\emptyset$, then $|L_u|+|F_u\cap \overline M|=|L_u|+|F'_u|\ge k+l+2$. By Claim \ref{g1}, part (a) follows as desired.

    Now suppose that $F'_u\cap M\neq\emptyset$. If there is a vertex $w\in F'_u\cap M$ such that $d_{L_u}(w)\le k$, without loss of generality, suppose that $u_1\in L_w$. Since $|L_w|=k+2$, there are two vertices $w',w''\in L_w\setminus L_u$. By Claim \ref{c}, $u_1w',u_1w''\in E$. It leads to $d(u_1)\ge |L_u\backslash \{u_1\}|+|\{u,w,w',w''\}|=k+5$, a contradiction.

    If there is a vertex $w\in F'_u\cap M$ such that $d_{L_u}(w)=k+1$ , without loss of generality, suppose $N_{L_u}(w)=\{u_1,u_2,\cdots,u_{k+1}\}$. Since $|L_w|=k+2$, there is a vertex $w'\in L_w\setminus L_u$. By Claim \ref{c}, $\{u_1,u_2,\cdots,u_{k+1},w'\}$ induces a clique in $G$. Then, $G[L_u\cup \{u,w,w'\}]$ is an $L$-configuration where $L=N_G(w)\cap L_u$, contradicting Claim \ref{p1}.

    Finally, we consider the case that there is a vertex $w\in F'_u\cap M$ such that $d_{L_u}(w)=k+2$. Let $F''_u=F'_u\setminus\{w\}$. If $F''_u\cap M\ne\emptyset$, let $w'\in F''_u\cap M$. By the above argument, we deduce that $d_{L_u}(w')=k+2$. Hence, $G[L_u\cup \{u,w,w'\}]$ is an $L$-configuration where $L= L_u$, contradicting Claim \ref{p1}. Now suppose $F''_u\subseteq \overline M$. If $|F''_u|=1$, let $F''_u=\{w''\}$ and we have $d_{L_u}(w'')=k+2$. Similar to the above proof, we obtain a contradiction. If $|F''_u|=2$, let $F''_u=\{w_1,w_2\}$ and $w_1,w_2\in \overline{M}$. Since $d_{L_u}(w_1)+d_{L_u}(w_2)=k+2$, without loss of generality, we assume that $d_{L_u}(w_1)\ge 2$. Since $|L_w|=|L_u|=k+2$, we obtain $|L_u|-d_{L_u}(w_1)\le k$, $uw_1\notin E$ and $ww_1\notin E$. By Claim \ref{g2}, we have proved part (a). If $|F''_u|\ge 3$, then $|L_u|+|F_u\cap \overline M|=|L_u|+|F''_u|\ge k+5$. By Claim \ref{g1}, part (a) follows as desired.
\end{proof}
\begin{claim}\label{m3}
    If there is a vertex $u\in \mathcal{U}$ such that $|L_u|=k+1$, part (a) follows as desired.
\end{claim}
\begin{proof}

    Suppose there is a vertex $u\in \mathcal{U}$ such that $|L_u|=k+1$. By Claim  \ref{c}, $L_u$ induces a clique in $G$. If $M\cap F'_u=\emptyset$, then $F'_u\subseteq \overline M$.  Since $G$ is a connected claw-free $(k+l+1)$-regular graph, $|N(u_1)\setminus(L_u\cup\{u\})|=k+l+1-|L_u|=l$, implying that $|F'_u|\ge l$. We claim that $|F'_u|\ge l+1$. Otherwise, suppose $F'_u=\{v_1,v_2,\cdots,v_l\}$, implying that $L_u\subseteq N_G[v_i]$ for each $i\in\{1,2,\cdots,l\}$. Then, $G[L_u\cup F_u]$ is an $L$-configuration where $L=L_u$, contradicting Claim \ref{p1}. So, $|F'_u|\ge l+1$ and $|L_u|+|F_u\cap\overline M|=|L_u|+|F'_u|\ge k+l+2$. By Claim \ref{g1}, part (a) follows as desired.

Now assume that $M\cap F'_u\neq\emptyset$. If there is a vertex $w\in M\cap F'_u$ such that $d_{L_u}(w)\le k-l+1$, without loss of generality, suppose that $u_1\in N_G(w)\cap L_u$. Since $|L_w|\ge k+1$, we have $|L_w\setminus L_u|\ge l$. Assume that $\{x_1,x_2,\cdots,x_l\}\subseteq(L_w\setminus L_u)$. By Claim \ref{c}, $u_1x_i\in E$ for each $i\in\{1,2,\cdots,l\}$. It leads to $d(u_1)\ge|L_u\setminus\{u_1\}|+|\{u,w,x_1,x_2,\cdots,x_l\}|\ge k+l+2$, a contradiction.

    Then, we suppose $d_{L_u}(w)\ge k-l+2$ for each $w\in M\cap F'_u$. If there exists a vertex $w_1\in F_u\cap \overline{M}$ such that $vw_1\notin E$ for each $v\in M\cap F_u$, by Claim \ref{g2}, part (a) follows as desired. Otherwise, we can assume that for each $w_1\in F_u\cap \overline{M}$, there is a vertex $v\in M\cap F_u$ such that $vw_1\in E$. By Claim \ref{c}, $N_G(v)\cap L_u\subseteq N_G(w_1)\cap L_u$, and so $d_{L_u}(w_1)\ge d_{L_u}(v)\ge k-l+2$. Hence, $d_{L_u}(w_1)\ge k-l+2$ for each $w_1\in F_u$. If $d_{L_u}(w)=k+1$ for each $w\in M\cap F'_u$, then for each $w'\in F'_u\cap \overline{M}$, there is a vertex $w''\in M\cap F_u$ such that $w''w'\in E$ and $d_{L_u}(w'')=k+1$. By the above argument, we deduce that $d_{L_u}(w')\ge d_{L_u}(w'')= k+1$ and $|F'_u|=l$. Then, $G[L_u\cup F_u]$ is an $L$-configuration where $L=L_u$, contradicting Claim \ref{p1}.

If there is a vertex $w\in M\cap F'_u$ such that $d_{L_u}(w)= k$, without loss of generality, suppose that $N_G(w)\cap L_u=\{u_1,u_2,\cdots,u_k\}$. Since $|L_w|\ge k+1$, there is a vertex $w_1\in L_w\setminus L_u$.   By Claim \ref{c}, $u_iw_1\in E$ for each $i\in\{1,2,\cdots,k\}$. Let $F''_u=F'_u\setminus \{w,w_1\}$. It is clear that $F''_u\neq\emptyset$. For $l=2$, let $w_2\in F''_u$. Then $d_{L_u}(w_2)=1<k=k-l+2$, contradicting that $d_{L_u}(x)\ge k-l+2$ for each $x\in F_u$.  For $l=3$, if there is a vertex $w_2\in F''_u$ such that $\{u_1,u_2,\cdots,u_k\}\subseteq N_G(w_2)\cap L_u$, we can similarly get a contradiction. Now we assume that for each vertex $v'\in F''_u$, $\{u_1,u_2,\cdots,u_k\}\not\subseteq N_G(v')\cap L_u$. If $F''_u\cap M\neq\emptyset$, suppose $w_2\in F''_u\cap M$. Since  $d_{L_u}(w_2)\ge k-l+2\ge k-1\ge l-1\ge 2$, we have $N_{L_u}(w)\cap N_G(w_2)\neq \emptyset$. Let $x\in N_{L_u}(w)\cap N_G(w_2)$. Since $d(x)=k+4$ and Claim \ref{c}, $\{u_1,u_2,\cdots,u_k\}\subseteq N_G(w_2)\cap L_u$, a contradiction. So, $F''_u\subseteq \overline M$. Let $y\in F''_u$. It is clear that $uy\notin E$. We claim that $wy\notin E$. Otherwise, suppose $wy\in E$. By Claim \ref{c}, $\{u_1,u_2,\cdots,u_k\}\subseteq N_G(y)\cap L_u$, a contradiction. Hence, $|L_u|-d_{L_u}(y)\le k$ and $vy\notin E$ for each $v\in M\cap F_u$. By Claim \ref{g2}, part (a) follows as desired.

    If there is a vertex $w\in M\cap F'_u$ such that $d_{L_u}(w)=k-1$, then we obtain $l=3$ since $d_{L_u}(w)=k-1\ge k-l+2$. Without loss of generality, assume that $N_G(w)\cap L_u=\{u_1,u_2,\cdots,u_{k-1}\}$. Since $|L_w|\ge k+1$, there are two vertices $w_1,w_2\in L_w\setminus L_u.$ By Claim \ref{c}, $u_iw_1,u_iw_2\in E$ for each $i\in\{1,2,\cdots,k-1\}$. Let $F''_u=F'_u\backslash \{w,w_1,w_2\}$. It is clear that $F''_u\neq \emptyset$. Then, for each $w'\in F''_u$, we have $d_{L_u}(w')\le 2$. Since $d_{L_u}(w')\ge k-l+2$ and $k\ge l$, we obtain $k=3$ and $d_{L_u}(w')=2$. If $F''_u\cap M=\emptyset$, then $F''_u\subseteq \overline M$. Let $z\in F''_u$. Then, $zu\notin E$. We claim that $zw\notin E$. Otherwise, suppose $zw\in E$. By Claim \ref{c}, $zu_1\in E$. It leads to $d(u_1)\ge |L_u\setminus\{u_1\}|+|\{u,w,w_1,w_2,z\}|\ge k+5$, a contradiction.  Since $|L_u|-d_{L_u}(z)\le k$ and Claim \ref{g2}, part (a) follows as desired. Then, we assume that $F''_u\cap M\neq\emptyset$ and $w_3\in F''_u\cap M$. If $w_1w_3,w_2w_3\in E$, then $d_{L_u}(w_1)=d_{L_u}(w_2)=4$ by Claim \ref{c}. So, $G[L_u\cup F_u]$ is an $L$-configuration where $L=L_u\cup \{w_1,w_2\}$, contradicting Claim \ref{p1}. If $w_1w_3,w_2w_3\notin E$, then there are two vertices $w_4,w_5\in L_{w_3}\setminus L_u$.  Since $w_3\in\mathcal{U}$ and Claim \ref{c}, we have $w_4,w_5\in F_u$. Then, $|L_u|+|F_u\cap\overline M|\ge |L_u|+|\{w_1,w_2,w_4,w_5\}|\ge k+l+2$. By Claim \ref{g1}, part (a) follows as desired. Now we consider the last case. Without loss of generality, suppose  $w_1w_3\in E$ and $w_2w_3\notin E$. Then, there is a vertex $w_4\in N(w_3)\setminus(L_u\cup \{w_1\})$ such that $w_4\in \overline{M}$. By Claim \ref{c}, $\{u_3,u_4,w_1,w_4\}$ induces a clique in $G$. So, $d(w_1)\ge |L_u|+|\{w,w_2,w_3,w_4\}|=8>k+l+1=7$, a contradiction.

\end{proof}
Since $|L_u|\in \{k+1,k+2\}$ for $l=2$ and $|L_u|\in \{k+1,k+2,k+3\}$ for $l=3$, by Claims \ref{m1}-\ref{m3}, part (a) follows as desired. Since $|V(G)|$ is finite, there exists an integer $q$ such that $P^\infty(S_q)=V(G)$. Hence, we complete the proof.
\end{proof}

We are now in a position to prove our main result, namely, Theorem \ref{MainResult}. %Recall its statement.

\begin{proof}
Let $G$ be a counterexample such that $|V(G)|$ is minimal. Let $S_0,S_1,\cdots,S_q$ be a sequence satisfying properties (a)-(b) in the statement of Lemma \ref{mainlma} with $q$ as small as possible. By Lemma \ref{mainlma} (b), the set $S_q$ is a $k$-PDS in $G$, and so $\gamma_{p,k}(G)\le |S_q|$. Since $S_0$ is a packing in $G$, we have that $|P^0(S_0)|=|N[S_0]|=(k+l+2)|S_0|$. If $q=0$, then $(k+l+2)|S_0|\le n$ and $\gamma_{p,k}(G)\le |S_0|\le \frac{n}{k+l+2}$, a contradiction. Now we suppose that $q\ge 1$. By Lemma \ref{mainlma} (a), $|S_q|=|S_0|+q$. By our choice of $q$, we decuce that $|P^\infty(S_{t+1})|\ge|P^\infty(S_t)|+k+l+2$ for $0\le t\le q-1$. Thus,
\[n=|P^\infty(S_q)|\ge |P^0(S_0)|+q(k+l+2)=(|S_0|+q)(k+l+2)=|S_q|(k+l+2).\]
Hence, $\gamma_{p,k}(G)\le |S_q|\le \frac{n}{k+l+2}$, a contradiction. This proves the desired upper bound.

Next, we show this bound is tight. For positive integers $k\ge l$ and $t$, we define the graph $C_{k,t}$ as follows. Take $t$ disjoint copies $C_i\cong A_l$ and link any two copies $(C_i,C_{i+1})$ with $l$ edges, where the subscripts are to be read as integers modulo $t$ and where $i=1,2,\cdots,t$. (see Figure \ref{equ}). Then, $C_{k,t}$ is a connected claw-free $(k+l+1)$-regular graph of order $n=t(k+l+2)$. Suppose that $S$ is an arbitrary $k$-PDS in $C_{k,t}$. It is easy to check that $C_i$ contains a $k$-fort of $G$, where $i=1,2,\cdots,t$. By Proposition \ref{k-fort}, $|S\cap V(C_i)|\ge 1$ for each $i\in\{1,2,\cdots,t\}$. It means that $\gamma_{p,k}(C_{k,t})\ge t=\frac{n}{k+l+2}$. Since the above proof, we obtain $\gamma_{p,k}(C_{k,t})\le \frac{n}{k+l+2}$. Hence,  $\gamma_{p,k}(C_{k,t})=\frac{n}{k+l+2}$.%If $S\cap V(C_i)=\phi$, then no vertex in $V(C_i)\setminus\{x_i,y_i\}$ belongs to the set $P^\infty(S)$, contradicting the assumption  that $S$ is a $k$-PDS in $C_{k,t}$. Therefore, $|S\cap V(C_i)|\ge 1$ for all $i$, where $1\le i\le t$.
\end{proof}

%\vskip 9em

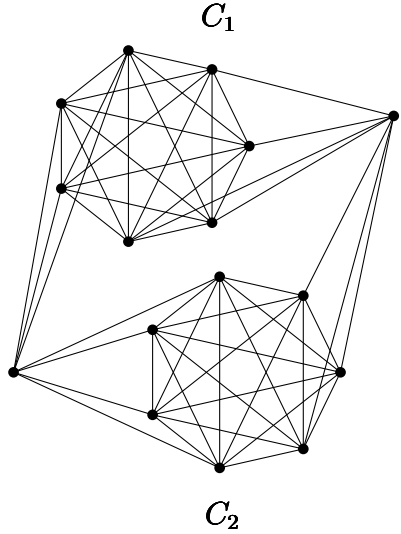
\begin{figure}[H]
    \centering
    \begin{tikzpicture}
        \foreach \i in {1,...,7}
        {
            \coordinate (x\i) at ({cos(360/7*\i)*1.3-0.6},{sin(360/7*\i)*1.3+1.5});
            \coordinate (y\i) at ({cos(360/7*\i)*1.3+0.6},{sin(360/7*\i)*1.3-1.5});
            \node at (0.3,3.2) {$C_1$};
            \node at (0.35,-3.4) {$C_2$};
            \fill (x\i) circle[radius=2pt];
            \fill (y\i) circle[radius=2pt];
        }
        \foreach \i in {1,...,6}
        {
            \foreach \j in {\i,...,7}
            {
                \draw (x\i) -- (x\j);
                \draw (y\i) -- (y\j);
            }
        }
        \coordinate (z1) at (-2.4,-1.5);
        \coordinate (z2) at (2.6,1.9);
        \fill (z1) circle[radius=2pt];
        \fill (z2) circle[radius=2pt];
        \foreach \i in {2,3,4}
        {
            \draw (z1) -- (x\i);
        }
        \foreach \i in {2,3,4,5}
        {
            \draw (z1) -- (y\i);
        }
        \foreach \i in {1,5,6,7}
        {
            \draw (z2) -- (x\i);
        }
        \foreach \i in {6,7,1}
        {
            \draw (z2) -- (y\i);
        }
    \end{tikzpicture}
    \caption{$C_{k,t}$ for $l=3$, $k=3$ and $t=2$}\label{equ}
\end{figure}
%\bibliographystyle{abbrv}
%\bibliographystyle{amsalpha}
%\bibliography{References}
\section{Conjecture and Question}
We pose the following conjecture which is still open.
\begin{conjecture}\label{ourConj}
For $l\ge 1$ and $k\ge l$, if $G$ is a connected claw-free $(k+l+1)$-regular graph of order $n$, then $\gamma_{p,k}(G)\le \frac{n}{k+l+2}$ and the bound is tight.
\end{conjecture}

Remark that if $l=1$, then the conjecture is true by the result of Chang et al. in \cite{chang2012generalized}. If $l\in\{2,3\}$, the conjecture is true by our Theorem \ref{MainResult}. When $l\ge 4$, the conjecture is still open. However, note that the bound of Conjecture \ref{ourConj} is tight since we can generalize the graph $C_{k,t}$ (defined in Section 3) to achieve this bound.

Now we pose the following question.

\begin{question}\label{CfreeQues}
For $r\ge 3$, let $G$ be a connected claw-free $r$-regular graph of order $n$. Determine the smallest positive value, $k_{min}(r)$, of $k$ such that $\gamma_{p,k}(G)\le \frac{n}{r+1}$.
\end{question}

By Observations \ref{odd} and \ref{even}, we deduce that $k_{min}(r)\ge \lfloor\frac{r}{2}\rfloor$. We remark that if Conjecture \ref{ourConj} is true, the answer of Question \ref{CfreeQues}  is $k_{min}(r)=\lfloor\frac{r}{2}\rfloor$. 
%\bibliographystyle{siamplain}
%\bibliographystyle{amsalpha}
%\bibliography{References}

\end{document}